\documentclass{amsart}
\usepackage{amsthm,amssymb,amsmath,hyperref,esint}
\usepackage{graphicx}
\usepackage{color}
\usepackage{tikz}
\usetikzlibrary{decorations.pathreplacing}
\usepackage[utf8]{inputenc}
\usepackage{textgreek}
\usepackage{mathtools}
\usepackage[margin=1.5in]{geometry}
\usepackage{xcolor}
\usepackage{parskip}
\usepackage{cleveref}

\def\N{\mathbb{N}}

\newtheorem{theorem}{Theorem}[section]
\newtheorem{lemma}[theorem]{Lemma}
\newtheorem{remark}[theorem]{Remark}
\newtheorem{definition}[theorem]{Definition}

\title{Infinite intersections of doubling measures, weights, and function classes}
\author{Theresa C. Anderson, David Phillips, Anastasiia Rudenko, Kevin You}
\date{September 26, 2024}
\begin{document}
\maketitle
\author
\begin{abstract}
    A series of longstanding questions in harmonic analysis ask if the intersection of all prime ``$p$-adic versions" of an object, such as a doubling measure, or a Muckenhoupt or reverse H\"older weight, recovers the full object. Investigation into these questions was reinvigorated in 2019 by work of Boylan-Mills-Ward, culminating in showing that this recovery fails for a finite intersection in work of Anderson-Bellah-Markman-Pollard-Zeitlin.  Via generalizing a new number-theoretic construction therein, we answer these questions.
\end{abstract}
\section{Introduction}
Dyadic techniques permeate much of mathematics, and in particular arise in analysis and number theory.  Restricting one's attention to numbers or intervals involving powers of $2$ can provide a key simplification to proving many theorems.  In particular, proving theorems for objects such as maximal functions, Muckenhoupt weights, and reverse H\"older weights whose definitions involve all intervals is more difficult than the corresponding dyadic versions.  Unfortunately, it is a classical result that dyadic versions of these objects are not equivalent to the full versions.  A longstanding folkloric set of questions asks if one intersects dyadic, triadic, and all prime $p$-adic (for any prime $p$) objects (where $p$-adic means that the intervals in the definition are restricted to lie in the $p$-adic grid - see Section 2 for precise definitions), does one recover the full object class?  Since doubling measures are so critical to analysis (underlying metric measure spaces, for example), we explicitly state this question for these: is
\[
\bigcap_{prime} Db_p = Db,
\]
where $Db_p$ indicates $p$-adic doubling measures and $Db$ doubling measures?  

Though many complimentary results have appeared (see, for example, \cite{Ward}), the authors are aware of little progress toward this question prior to 2019, outside of unpublished work of Peter Jones discussed in \cite{Krantz}.  Then in 2019, Boylan-Mills-Ward made substantial progress by showing that dyadic and triadic doubling do not imply doubling (hence, at a minimum, more primes would be needed to make the above claim true).  Precise definitions of all of these concepts appear in Section 2.  Their work \cite{BMW} involved careful case analysis, and a significant dose of elementary number theory highly specific to the \emph{bases} $2$ and $3$ that did not generalize easily.  Inspired by their progress, Anderson and Hu expanded their result to any two primes $p$ and $q$.  This involved several significant new steps, including a much heavier interplay of number theoretic methods.  Indeed, their progress therein allowed them to disprove a number theoretic conjecture of Krantz \cite{Krantz}.  Additionally, they were able to apply their results to answer corresponding questions with regard to $A_p$ and reverse H\"older weights \cite{AH}.  Anderson-Travesset-Veltri generalized this result to coprime bases and extended applications to the function class $BMO$ \cite{ATV}.  Finally, with entirely new techniques, the authors of \cite{ABMPZ} extended all claims for measures, weight and function classes to a finite set of bases, that is, that the intersection of all $n$-adic objects, for a finite set of $n\in \mathbb{N}$, is never equal to the full object.  This involved a completely new construction technique, relying on an interplay of topology, number theory, analysis and geometry that has deep connections with unsolved conjectures, such as Schanuel's conjecture \cite{Sch}.  Inspired by this new construction, we are able to fully answer this folkloric conjecture.

Firstly, we show that when one considers the whole real line, one can construct a measure that is $n$-adic doubling for all $n\in\mathbb{N}\setminus \{1\}$, but that is not doubling overall.  This indicates why previous work focused on measures defined on the torus $\mathbb{T} = [0,1]\setminus \{0\equiv1\}$ and alludes to some of the difficulties in this restriction (notably, the lack of an ``origin").  Then we are able to push the technical limits of the construction from \cite{ABMPZ} to show our first main result:

    \begin{theorem}\label{main}
        There exists an infinite family of measures
        on $\mathbb{R}$ and on $\mathbb{T}$
        that are $n$-adic doubling for each $n \geq 2$, yet not
        doubling overall. 
    \end{theorem}

    Moreover, we carefully investigate the doubling constant
    involved in our construction and show that each $n$-adic doubling constant can be made to grow
    arbitrarily slowly in $n$,  but it cannot be independent of $n$.  We will call a measure uniformly $n$-adic doubling if the $n$-adic doubling constant can be chosen to be universal simultaneously for all $n \in \mathbb{N}\setminus \{1\}$.  It turns out that these measures must be doubling overall, the subject of our next result.  

    \begin{theorem}\label{uniform}
    Let $\mu \neq 0, \infty$ be a measure on $\mathbb{R}$ or on $\mathbb{T}$ that is uniformly $n$-adic doubling for all $n \in \N$.  Then $\mu$ is doubling overall. 
    \end{theorem}
    
    Using our construction from Theorem \ref{main}, we can obtain results on the intersection of weight and function classes. This improves previous results in the field from finite to infinite intersections, for a wide variety of objects including reverse H\"older and Muckenhoupt weights, the bounded mean oscillation (BMO) function class, Hardy spaces and the vanishing mean oscillation (VMO) function class.  These are natural and useful classes to consider as they permeate modern harmonic analysis.  Precise definitions of these classes are given in Section 5.

    \begin{theorem}\label{app}
       The following hold:
        \begin{enumerate}
            \item $\bigcap_{n \geq 2} RH^n_r \subsetneq RH_r$ for  $1 \leq r < \infty$
            \item $\bigcap_{n \geq 2} A^n_r \subsetneq A_r$ for  $1 < r \leq \infty$
            \item $\bigcap_{n \geq 2} BMO_n \subsetneq BMO$
            \item $\bigcap_{n \geq 2} H^1_n \subsetneq H^1$
            \item $\bigcap_{n \geq 2} VMO_n \subsetneq VMO$
        \end{enumerate}

    \end{theorem}

Though inspired by recent advances in the area, these results require significantly new ideas.  Since the previous constructions already involve a subtle iteration, we had to carefully introduce a new iteration scheme that would seamlessly integrate into the former setup.  This new technical advancement iterates the construction in \cite{ABMPZ} a countably infinite number of times, with each iteration carefully placed to interplay with the next iteration.  The care to ensure that all the constants behave in the proper way to ensure $n$-adic doubling yet not doubling requires intricate computation and precise geometric bounds to work in tandem with the extensive number theory involved in previous works.  Due to this care, we were able to track the $n$-adic doubling constants and show the limit of this construction in Theorem \ref{thm3}.  In particular, a specific key difference between ours and previous works pertains to a parameter $\alpha$.  In past constructions, $\alpha$ would cycle through the natural numbers.  Now we allow $\alpha$ to be a more general sequence of natural numbers that grows to infinity, and, more substantially, $\alpha$ is chosen dynamically -- depending on the interval where it was originally defined.  Much of the motivation for this new construction is outlined in Sections 3 and 4, beginning with an illustrative heuristic.

This paper is organized as follows.  Section 2 provides background and definitions and concludes with a guide to how our three main results, Theorem \ref{main}, Theorem \ref{uniform} and Theorem \ref{app}, are proved.  Section 3 starts with a detailed heuristic about constructing measures that are $n$-adic doubling for all $n$ yet not doubling, and ends with several constructions of such measures, showing Theorem \ref{uniform}.  Section 4 centers around the proof of Theorem \ref{main}, which is accomplished via a series of steps, presenting the intricate iterated nature of the construction while simultaneously reviewing and generalizing the main construction in \cite{ABMPZ}.  Finally, we analyze the \emph{doubling constants} from our construction, culminating in Theorem \ref{thm3}.  The final section, Section 5, includes all the applications, which take significant care, resulting in the proof of Theorem \ref{app}.  Detailed definitions of the weight and function classes appear here along with connections to previous work.  In our proofs we attempt to use the letter $J$ for a $n$-adic interval and $I$ for other intervals (including arbitrary ones).

\subsection{Acknowledgements}
This research is supported by National Science Foundation DMS grants 2231990, 2237937 (TCA).  The authors thank David Cruz-Uribe and Lesley Ward for helpful feedback related to this work.  
\footnote{Corresponding author (TCA): Carnegie Mellon University, tanders2@andrew.cmu.edu. Authors DP and AR: University of Pittsburgh, KY: Carnegie Mellon University.  Keywords: doubling measures, weights, dyadic harmonic analysis}    
    \section{Preliminaries}
    We begin by recalling some definitions, but keep this section brief and refer the interested reader to the many excellent references on doubling measures, dyadic harmonic analysis, and related concepts, some of which are listed in the bibliography (\cite{Cruz}, \cite{LPW}, \cite{P}).  A simple motivating example appears at the end of the section, providing an introduction to the subtleties addressed in the proofs of our first two theorems.  We conclude with a guide to proving the main results at the end of this section.
    \subsection{Definitions}

    Throughout this paper we work on the real line or the torus with the points 0 and 1 identified.  When we work with intervals contained in $[0,1]$ it will be unimportant whether the ambient space is $\mathbb{R}$ or $\mathbb{T}$.  Much of the focus in this paper will be on intervals of this type.  For our purposes (unless otherwise stated) 
    all intervals we refer to are bounded half-open 
    intervals with a left endpoint. We will often work with a distinguished collection of \emph{n-adic} intervals, defined below, which have a prescribed structure.

    \begin{definition}
        A doubling measure $\mu$ is a measure for which
        there exists a positive constant $C$ such that
        for any adjacent intervals
            $I_1, I_2$ of equal length we have
             \[ \mu(I_1) \leq C \mu(I_2). \]
        Equivalently,
\[
\frac{1}{C} \leq \frac{\mu(I_1)}{\mu(I_2)} \leq C.
\]
    \end{definition}

    Note that this constant $C$ must not depend on $I_1, I_2$, though it
    will depend on $\mu$.  We now define the more restrictive $n$-adic intervals.  When $n=2$ these are the dyadic intervals. 

    \begin{definition}
        For a positive integer $n \geq 2$, the standard $n$-adic system $\mathcal{D}(n)$ is the collection of $n$-adic intervals in $\mathbb{R}$ of the form
        \[ I=[\frac{k-1}{n^m}\frac{k}{n^m}) \text{ for } m,k\in\mathbb{Z}. \]
    \end{definition}

    \begin{definition}
        For an $n$-adic interval $I=[\frac{k-1}{n^m}\frac{k}{n^m})$, the $n$ children of $I$ are given by: 
            \[ I_j=[\frac{k-1}{n^m}+\frac{j-1}{n^{m+1}},\frac{k-1}{n^m}+\frac{j}{n^{m+1}}) \text{ for } 1\leq j\leq n.\] 
        
    \end{definition}
        The definition is slightly different on the torus due to the lack of large $n$-adic intervals, but we will not need to consider those here. 
        
        We say that the $n$-adic intervals $J_1$ and $J_2$ are $n$-adic siblings if they are the children of some common $n$-adic interval $J$. 
        It is
        necessary but not sufficient for $J_1, J_2$ to have the same length. 
        
    \begin{definition}
        An $n$-adic doubling measure $\mu$ is a measure for which
        there exists a positive constant $C(n)$ such that
        for any $n$-adic siblings $J_1, J_2$ we have 
        \[ \mu(J_1) \leq C(n) \mu(J_2). \]
        or equivalently,
        \[
        \frac{1}{C(n)} \leq \frac{\mu(J_1)}{\mu(J_2)} \leq C(n).
        \]
    \end{definition}

    Note that the domain of $\mu$ may be the whole real line, but oftentimes will be restricted to a subset, such as $[0,1]$, in ours and related constructions.  when the domain of $\mu$ is not an interval with integer endpoints, we ignore $n$-adic intervals that only partially lie in $\mu$'s domain.
    
        We pause to remark on some important distinctions between intervals on $\mathbb{R}$ and $\mathbb{T}$.  First, let's discuss $n$-adic intervals.  In both of these domains, $n$-adic intervals of length at most $1$ behave the same.  In particular, intervals of this type in $[0,1]$ and $\mathbb{T}$ are identical.  Moreover, on the torus there are no larger $n$-adic intervals due to the identification of the points $0$ and $1$.  Furthermore, the measures we define function as Lebesgue measure outside of a bounded set, so $n$-adic intervals of length larger than $1$ will simply inherit Lebesgue measure and therefore the measure will be automatically $n$-adic doubling on these intervals.  Thus for all of our main results, Theorems \ref{main}, \ref{uniform} and \ref{app}, we only need to focus on $n$-adic intervals of length at most $1$, handling $\mathbb{T}$ as $[0,1]$, and requiring no extra work.

        Now, let's discuss arbitrary intervals.  To show a measure is not doubling, it suffices to find an infinite sequence of pairs of adjacent intervals $(G,H)$ with $|G| = |H| <1$ such that the ratio $\frac{\mu(G)}{\mu(H)}$ is unbounded.  We will hence only be analyzing small intervals inside $[0,1]$ (or analogous domains), making again the case $\mathbb{R}$ and $\mathbb{T}$ identical.  There are two potential subtleties.  One potential difference between $\mathbb{R}$ and $\mathbb{T}$ occurs in Theorem \ref{uniform} with intervals of any length that contain $0$, where working on $\mathbb{T}$ is potentially more restrictive due to the identification.  However, we will show doubling on $\mathbb{R}$ which will imply doubling for $\mathbb{T}$.  The second difference occurs in the definition of the space $VMO$ in Theorem \ref{app}, but since we use duality of function spaces in this claim (which is retained with the stated different definitions on $\mathbb{R}$ versus $\mathbb{T}$), we will not need to use the definition directly.  This remark illustrates the subtle differences between $\mathbb{R}$ and $\mathbb{T}$ that were taken into consideration both here and in preceding results.

\subsection{Motivating example}
    We begin with the first observation that if the domain $\mu$ contains an integer point in the interior, 
    then we have a straightforward construction to Theorem \ref{main} for $\mathbb{R}$.  This also explains why prior work (as well as ours) focused on measures defined in $[0,1]$ or $\mathbb{T}$. 
    
    Consider the absolutely continuous measure $\mu$ with domain $[-1,1)$ with density defined as 
    \[ f := \frac{d\mu}{d\mathcal{L}} = 
    \begin{cases} 1 \text{ if } x < 0, \\ x \text{ if } x > 0 \end{cases}. \]
    Then easily $\mu$ is $n$-adic doubling with constant $C(n) = 2n-1$, 
    attained when $J_1 = [(n-1)/n^k, n/n^k)$ and $J_2 = [0, 1/n^k)$ for $k \geq 1$. 
    However, $\mu$ is not doubling, for $\mu$ crosses $0$, and doubling can detect this.  This example illustrates the difference between arbitrary intervals, where 0 can be an interior point, and $n$-adic intervals, where 0 is never an interior point.  It also motivates the study of \emph{adjacent dyadic systems}, another intriguing area of study in modern analysis; however, we do not consider these here.

    Another similar construction can be made for the density
    \[ f := \frac{d\mu}{d\mathcal{L}} = 
    \begin{cases} 1 \text{ if } x < 0, \\ 2^{-k} \text{ if } x \in [2^{-k},2^{-k+1}), k \geq 0 \end{cases}. \]
    This infinite staircase example is also $n$-adic doubling and not doubling, 
    and the doubling constant is also attained at $J_1 = [(n-1)/n^k, n/n^k)$ and $J_2 = [0, 1/n^k)$ for appropriate $k$. 
\subsection{Plan for proof of main results}
We provide a brief guide as to how our main results are proved.  Section 3 provides a heuristic behind Theorem \ref{uniform} and then we prove Theorem \ref{uniform} via two lemmas: Lemma \ref{3.1} and Lemma \ref{3.2}.  Section 4 begins by recalling a key result from the construction of \cite{ABMPZ} followed by the proof of Theorem \ref{main} via three steps: Theorems \ref{thm1}, \ref{thm2} and \ref{thm3}.  Finally Section 5 begins the proof of Theorem \ref{app} via Lemma \ref{RH aux lemma} and Theorem \ref{RH thm}, which prove the first part of Theorem \ref{app} (the claim for reverse H\"older weights{.  This is followed by the (short) proofs of the other claims, which follow in a streamlined fashion from part 1 of Theorem \ref{app}.

    \section{$n$-adic doubling constants and uniformity}
        Note that is easy to decrease the $n$-adic doubling constant
    $C(n)$ from the previous example by changing $x$ to $x^m$ for some $0 < m < 1$. 
    Nevertheless, $C(n)$ will depend on $n$. 
    One may wonder if $C(n)$ can made to be independent of $n$.
    We show that this is not possible, regardless of the 
    domain.  The goal of this section is provide motivation behind Theorem \ref{uniform}, connecting this result in a variety of ways to others in the field, and eventually providing an expedient proof.    

    \subsection{Motivation}

We first provide a heuristic as to why the $n$-adic constant for any weighted measure must depend on $n$ unless it is also doubling.  We also outline connections between this heuristic and three previous papers in the area, namely \cite{Krantz}, \cite{AH} and \cite{ABMPZ}.  To match notation, we consider $n=p$ for $p$ prime, but everything observed here works for general $n$.

The authors in \cite{AH} disproved a conjecture of Krantz \cite{Krantz} by proving the following: 

    Fix $C>0$.  For every $m\in\N$ there exists a $C(m)$ for every natural number $\beta$ and $p$ prime such that for $\frac{1}{C}\leq\frac{p^n}{2^m}\leq C$
    \begin{equation}
    \label{separation_similiar_scales}
        \bigg|\frac{1}{2^m}-\frac{\beta}{p^n}\bigg|\geq\frac{C(m)}{2^m}.
    \end{equation}

Recently, by choosing the sequence $n=[mlog_{p} 2]$, $\beta=1$, \cite{ABMPZ} showed that for every $\epsilon>0$: 

\begin{equation}
\label{closeness}
    \bigg|\frac{1}{2^m}-\frac{1}{p^n}\bigg|<\frac{\epsilon}{2^m}
\end{equation}
holds for infinitely many $m$. 

In their notation, \eqref{closeness} becomes the assertion that $|Y-Z|<\epsilon|J|$, where $I$ is a dyadic interval of size $\frac{1}{2^m}$, $Y$ is an endpoint of one of the children of $I$, and $Z$ is an endpoint of a $p$-adic interval $J_{p_1}$ of size $\frac{1}{p^n}$.  In other words, distinguished points of dyadic and $p$-adic intervals line up very nicely, within an $\epsilon$-window when scaled to the length of $I$.  Precise details are reviewed in the extension that we pursue in the next section. 

Since we have \eqref{closeness}, we can examine the ratio $\frac{\mu(J_{p_1})}{\mu(J_{p_2})}$ used in calculating the $p$-adic doubling constant for any measure $\mu$ that is a reweighting of Lebesgue measure.  We first look at this ratio in terms of Lebesgue measure:

\[\frac{|J_{p_1}|}{|J_{p_2}|} = \frac{|I+\widetilde{\epsilon}I|}{|J_{p_2}|}\simeq\frac{\frac{1}{2^m}}{\frac{1}{p^n}}=\frac{p^n}{2^m},
\]

where $\widetilde{\epsilon}\leq\epsilon$ and $\simeq$ indicates equality up to this $\epsilon$.

If we also assume $p$-adic doubling for some universal $C$, independent of $p$, we have 
\[
\frac{1}{C}\leq\frac{\mu(J_{p_1})}{\mu(J_{p_2})}\leq C.
\]
 We run into an issue if we want to consider an infinite set of $p$, as for any newly chosen $\epsilon>0$, $\frac{p^n}{2^m}\geq p^{\epsilon}$ or $\frac{p^n}{2^m}\leq\frac{1}{p^{\epsilon}}$ infinitely often.  Thus in rewriting the $\mu$-ratio in terms of its weights and Lebesgue measure, $p^\epsilon$ terms will always appear. Moreover, if Schanuel's conjecture is true, we could get comparison constants of ${p^{\frac{1}{2}-\epsilon}}$, contradicting our assertion of a universal $C$ for all $p$. In other words, even in terms of the Lebesgue measure, the ratio for $p$-adic intervals that we want to compare must depend on $p^{\epsilon}$ infinitely often.  This factor of $p^\epsilon$ cannot be offset by any fixed reweighting if the set of $p$ is also infinite.

For an finite set of $p$, this observation is not problematic as the doubling constant can be bounded by $C$: $(\max p)^{\frac{1}{2}}\leq C$ (such as in \eqref{separation_similiar_scales} where $\frac{p^n}{2^m}\leq C$ is assumed). Moreover, these observations easily generalize from dyadic to $l$-adic and $p$-adic to $n$-adic intervals when $n$ not a multiple of $l$.  Therefore a weighted measure that is $n$-adic doubling for every $n$ must be either also doubling itself or have each $n$-adic doubling constant dependent on $n$.

\subsection{Proof of Theorem \ref{uniform}}
We split the proof to two steps, first for $\mu$ on $[0,1)$, and then $\mu$ on $\mathbb{R}$.  We briefly remark on considerations for the torus at the end of the second lemma; hence these lemmas will immediately give Theorem \ref{uniform}. 

\begin{lemma}
\label{3.1}
    Let $\mu$ defined on $[0,1)$ be uniformly $n$-adic doubling. 
    Then $\mu$ is also doubling with the same constant. 
\end{lemma}

\begin{proof}
    
Let $\mu$ be a measure which is uniformly $n$-adic doubling for arbitrarily large $n$ with constant $M$. Let $I_1$ and $I_2$ be arbitrary intervals
of the same length (they need not be adjacent).

Choose $n$ large such that $\mu$ is $n$-adic doubling. Let $z = \lceil n / \vert I_1 \vert
\rceil$. There are at least $(z - 2)$ $n$-adic intervals
that are the children of $[0,1)$ fully contained within $I_1$. Call any such child $A_i$. 
Likewise, there are at most $z$ $n$-adic intervals that are the children of $[0,1)$
that cover $I_2$. Call any such child $B_i$.

Since for any $i,j$ we have that $A_i$ and $B_j$ are siblings,
we can bound $\mu(B_j) \leq M \mu(A_i)$. 
Thus 
\[ \mu(I_2) \leq \sum_{i=1}^{z} \mu(B_i) \leq
\frac{Mz}{z-2} \sum_{i=1}^{z-2} \mu(A_i) \leq \frac{Mz}{z-2} \mu(I_1). \]
By taking $n \rightarrow \infty$, we get $z \rightarrow \infty$ and
conclude that $\mu$ is doubling also with constant $M$.
\end{proof}

    \begin{lemma}
    \label{3.2}
    Let $\mu$ defined on $\mathbb{R}$ be uniformly $n$-adic doubling,
    and $0 < \mu([-1,0)), \mu([0,1)) < \infty$.
    Then $\mu$ is also doubling. 
    \end{lemma}
    
\begin{proof}
     Let $\mu$ be uniformly $n$-adic doubling with constant $M$. 
     Then, $\mu([\eta,\eta+1))$ is bounded above over $\eta \in \mathbb{N}$.
     From the previous theorem applied to $\mu$ restricted to $[\eta,\eta+1)$,
     since $\mu$ is doubling, 
     we can actually conclude that $\mu$ is absolutely continuous. 
    Let $f$ be the density of $\mu$.
    
    Suppose that $x,y$ are Lebesgue points of $f$ on $[\eta, \eta+1)$. 
    Then, $f(x) = \lim_{r \rightarrow 0} \mu(B(x,r))$ and likewise for $y$.
    However, $\mu(B(x,r))$ can be approximated by $n$-adic intervals for sufficiently large $n$,
    so we have that $\mu(B(x,r)) \leq M \mu(B(y,r))$, from the previous theorem. 
    Letting $r \rightarrow 0$ yields that $f(x) \leq M f(y)$.
    Therefore, the range of $f$ on its Lebesgue points must be contained within some range $[C, MC]$ for $C > 0$, and $f$ is doubling with constant $M$.
    
    Suppose that $\mu([\eta,\eta+1)) = A_\eta$. Then, by the previous result the range of $f$
    is contained within $[M^{-1} A_\eta, MA_\eta]$. However, the set of all $A_\eta$ for $\eta \in \mathbb{N}$
    must be contained within some range $[C,MC]$.  By applying $n$-adic doubling for sufficiently large $\eta$,
    the range of $f$ is contained in $[M^{-1} C, M^2 C]$ on $[0,\infty)$,
    so $f$ is doubling with constant $M^3$ on $[0,\infty)$. 
    
    The same argument can be applied to $\mu$ on $(-\infty,0)$, to yield that $\mu$ is
    bounded above and away from zero again in a range difference of $M^3$. Thus
    $f$ is essentially bounded above and away from zero on $\mathbb{R}$,
    and therefore doubling. However, we cannot control the exact doubling constant,
    since the ratio of $\mu([-1,0))$ to $\mu([0,1))$ can 
    take any finite value.  On the torus we automatically have that $\mu([0,1]) = \mu([-1,0])$, so the assumption in the lemma is automatic.  Additionally, the proof can be analogously reworked under the assumption that $0<\mu[-c,0]),\mu([0,c]<\infty$ for any $c$, thus leading the the assumption that $\mu \neq 0,\infty$.  These observations underscore the well-known importance of the origin to the standard $n$-adic grids.
\end{proof}

    \section{Proof of Theorem \ref{main}}
We will prove Theorem \ref{main} in three stages -- Theorems \ref{thm1}, \ref{thm2}, and \ref{thm3} below.  We begin by recalling a key lemma from \cite{ABMPZ} that provides the necessary number theoretic structure in our construction.  This lemma was already alluded to in the previous section.

    \begin{lemma} \label{Lemma4}
    Let ${n_1,...,n_k}\subseteq\mathbb{N}$, $\alpha\in\mathbb{N}$ and $\epsilon=2^{-100\alpha}>0.$ Then, there exists infintely many $x\in\mathbb{N}$ such that for all $1\leq i\leq k$ that 

    \[ \Bigg| \frac{1}{2^x}-\frac{1}{n_i^{[xlog_{n_i}{2}]}} \Bigg| <\frac{\epsilon}{2^x} \]
    \newline
    \end{lemma}

Note that there is a small error in the arXiv preprint version of \cite{ABMPZ}, which does not appear in the published version.  Lemma 3.8 there is a claimed extension of the lemma above to an infinite set of $n_i$, due to to using an incorrect metric.  The correct metric is used for all other claims in the paper, and therefore does not impact any of their main results.  This observation motivated our construction below.

While the main ideas present in the work \cite{ABMPZ} still lead to the framework in the proof of Theorem $\ref{main}$, our new insight is a subtle iteration to carefully track the doubling constants and also add flexibility to defining the measures.  In particular, the sequence of $\alpha$ that is used in the construction can differ on each interval $[\eta - 1, \eta)$ (these intervals no longer need to be linked to the $n$ in the $n$-adic construction in a linear fashion).  Thus, we will be updating $\alpha = \alpha_{\eta}$ for each $\eta$, a key difference from previous works.  Though this may seem like a minor point, it is actually the crux of our argument, and allows one to push the new methods in previous works to their limit.  We begin with the base step of our construction before moving it to different intervals $[\eta - 1, \eta)$

\begin{theorem}\label{thm1}
    Let $0<a<1$ and $1<b<2$ with $a+b=2$ and define $\kappa = b/a >1$.  For any natural numbers $\alpha, M > 1$, 
    there exists infinitely many measures $\mu$ on $[0,1)$ such that
    \begin{itemize}
        \item there exists two adjacent intervals $I_1, I_2$ with $\frac{\mu(I_1)}{\mu(I_2)} = \kappa^\alpha$
        \item $\mu$ is $n$-adic doubling with the constant $1.01 \kappa^3 (2n)^{ \log_2 \kappa}$ for $n \leq M$
        \item if $f$ is the density of $\mu$, then $f$ is bounded above and away from zero, and the ratio $\max f / \min f \leq \kappa^\alpha$.\\
    \end{itemize}
\end{theorem}

Note that the constant $1.01$ is not important and can be changed to any constant larger than $1$ by adjusting $\epsilon$ appropriately. 

\begin{proof}
To prove this, we will first choose primary construction constants, and then establish a re-weighting scheme to construct $\mu$. Then we show that our measure is $n$-adic doubling with a given constant, by considering all possible cases of intervals.  While we focus on our new steps compared to previous works, we attempt to present a self-contained proof.

Let $\epsilon = 2^{-100\alpha}>0$, and consider a finite set of integers $ \{ 2, 3, \ldots, M \}$. Then from Lemma \ref{Lemma4} above, there exists a natural number $x$ such that
\[ \Bigg| \frac{1}{2^x} - \frac{1}{n^{[x \log_{n} 2]}} \Bigg| < \frac{\epsilon}{2^{x}}, \;\;\;\;  2 \leq n \leq M \]

To simplify, define the following notation:
\begin{itemize}
    \item a sequence $p_n = [x \log_{n} 2]$;
    \item a sequence $Y_n = \frac{1}{n^{p_n}}$;
    \item $Z = \frac{1}{2^x}$
\end{itemize}
We can now re-write the above inequality to be:
\[ \vert Z - Y_n \vert < \epsilon Z.\]
We are ready to construct $\mu$.

\hypertarget{re-weighting}{\textbf{Re-weighting:}}
Given a measure $\mu$ with density $f$, suppose that 
on a dyadic interval $I$, $f$ is constant with value $f_0$.
We say that $I$ splits its weight $a$ to
its left child and $b$ to its right child to mean $f = a f_0$ on
the left child and $f = b f_0$ on the right child.
Since $a+b = 2$, the value of $\mu(I)$ is preserved. 

Start with $\mu$ to be the uniform Lebesgue measure on $[0,1)$. 
We perform $2\alpha$ steps of rew-eighting on $[Z/2,3Z/2)$. 
For step $\beta$, where $\beta \leq \alpha$, the rightmost $\beta$-generation descendant 
 of $[0, Z)$ splits weight $a$ to its left child and $b$ to its right child, and likewise for the leftmost $\beta$-generation descendant of $[Z,2Z)$. We do the same for $\alpha < \beta \leq 2\alpha$ but with $a$ and $b$ swapped, meaning that the rightmost $\beta$-descendant
 of $[0,Z)$ and leftmost $\beta$-descendant of $[Z,2Z)$ split weight $b$ to the left child and $a$ to the right child. See Figure \ref{fig:reweighting}. 
This defines  $\mu$ and $f$, with $f$ bounded above by $b^\alpha$ and below from $a^\alpha$
and $f$ is identically $1$ outside $[Z/2,3Z/2)$. If we take the leftmost and rightmost $\alpha$-th generation descendants, we find such two adjacent intervals $H, G$ such that $\frac{\mu(H)}{\mu(G)} = \kappa^\alpha$.

\begin{figure}[!ht]
    \centering
\begin{tikzpicture}[scale=0.4]
\def\e{0.2} 
\def\f{1.5} 
\def\g{-0.8} 
\draw (0,0) -- (32,0);
\draw (0,3*\e) -- (0,-3*\e);
\draw (32,3*\e) -- (32,-3*\e);
\draw (16,3*\e) -- (16,-3*\e);
\draw (8,\e) -- (8,-\e);
\draw (12,\e) -- (12,-\e);
\draw (14,\e) -- (14,-\e);
\draw (15,\e) -- (15,-\e);
\draw (24,\e) -- (24,-\e);
\draw (20,\e) -- (20,-\e);
\draw (18,\e) -- (18,-\e);
\draw (17,\e) -- (17,-\e);

\draw (0,\f) node {$Z/2$};
\draw (32,\f) node {$3Z/2$};
\draw (16,\f) node {$Z$};

\draw (4,\g) node {$a$};
\draw (10,\g) node {$ba$};
\draw (13,\g) node {$b^3$};
\draw (12.5,3*\g) node {$b^2ab$};
\draw (15,3*\g) node {$b^2a^2$};
\draw (14.5,\g) -- (13, 2.2*\g);
\draw (15.5,\g) -- (15, 2.2*\g);

\draw (28,\g) node {$a$};
\draw (22,\g) node {$ab$};
\draw (19,\g) node {$a^3$};
\draw (19.5,3*\g) node {$a^2ba$};
\draw (17,3*\g) node {$a^2b^2$};
\draw (17.5,\g) -- (19, 2.2*\g);
\draw (16.5,\g) -- (17, 2.2*\g);
\end{tikzpicture}
    \caption{Density due to re-weighting procedure with $\alpha=2$} 
    \label{fig:reweighting}
\end{figure}
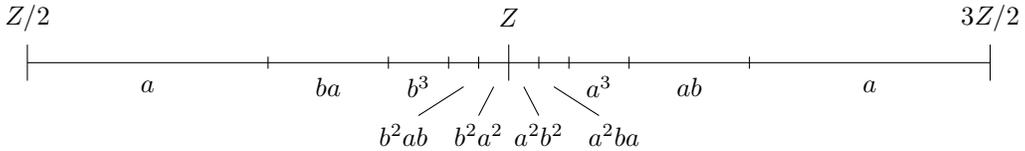

Our re-weighting\hyperlink{re-weighting}{$^{\#}$} splits $[Z/2,3Z/2)$ into $4\alpha + 2$ intervals, where $f$ is constant on each interval. Call these intervals \emph{nice intervals}.
In order, each interval has length 
\[ 2^{-1}Z, 2^{-2}Z, \ldots, 2^{-2\alpha}Z, 2^{-2\alpha}Z, 
\ldots, 2^{-2}Z, 2^{-1}Z. \]  
Moreover, the value of $f$ on the nice intervals are
\[ a, ba, b^2a, \ldots, b^{\alpha-1}a, b^{\alpha+1}, b^{\alpha}ab, b^{\alpha}a^2b, \ldots, b^{\alpha} a^{\alpha-1} b, b^\alpha a^\alpha \]
on the first $2\alpha+1$ nice intervals, and with $a,b$ switched and the list inverted on the next $2\alpha+1$ nice intervals.

\begin{remark}
     Note that on adjacent nice intervals,
the ratio of values of $f$ are
\[ b, \ldots, b, \frac{b^2}{a}, a, \ldots, a, \frac{a}{b}, 1, \frac{b}{a}, b, \ldots, b, \frac{a^2}{b}, a, \ldots, a.\] 

\end{remark}

Thus, on two adjacent intervals the ratio of $f$ is at most $\kappa^2$
(in fact at most $\kappa$ if we ignore $b^2/a$ or $a^2/b$),
and on $t$ adjacent intervals,
$\max f / \min f$ has a ratio of at most $\kappa^{t}$. 
This is a very generous estimate, and the bound is never attained.  See the construction in \cite{AH} for more details on how these values arise.

Following this, we show that $\mu$ is $n$-adic doubling for $n \leq M$ with the appropriate constant. In the trivial case when $n=2$, the above statement holds because $\mu$ is a $2$-adic doubling with constant $\max(b, 1/a) < \kappa$ by construction. 
Now, consider $n > 2$. With that, let $J = (\frac{y}{n^p}, \frac{y+1}{n^p})$ be an $n$-adic interval.  There are a few cases.\\

\textbf{Case 1:} $p < p_n - 1$. In this case, $\vert J \vert > n Y_n$, so any children of $J$ have length at least $n Y_n$, which
means it either contains $(Z/2,3Z/2)$ or is disjoint from it. None of the children 
``sees'' the re-weighting, ao all children have $\mu$-measure equal to their length. \\

\begin{figure}[!ht]
    \centering
    \begin{tikzpicture}[scale=1.3]
\def\e{0.1} 
\def\eps{0.07} 
\def\f{0.4} 
\draw (0,0) -- (9,0);
\draw (0,2*\e) -- (0,-2*\e);
\draw (9,2*\e) -- (9,-2*\e);
\draw (3,2*\e) -- (3,-2*\e);
\draw (6,2*\e) -- (6,-2*\e);
\draw[blue] (3-\eps,2*\e) -- (3-\eps,-2*\e);
\draw[blue] (1.5-0.5*\eps,\eps) -- (4.5-1.5*\eps,\eps);
\draw (1,\e) -- (1,-\e);
\draw (2,\e) -- (2,-\e);
\draw (4,\e) -- (4,-\e);
\draw (5,\e) -- (5,-\e);
\draw (7,\e) -- (7,-\e);
\draw (8,\e) -- (8,-\e);
\draw (0,-\f) node {$0$};
\draw (9,-\f) node {$nY_n$};
\draw (3+\eps,-\f) node {$Y_n$};
\draw (1,-\f) node {$n^{-1}Y_n$};
\draw[blue] (2,0.7*\f) node {Re-weighting};
\draw[blue] (3-\eps,\f) node {$Z$};
\end{tikzpicture}
    \caption{Triadic intervals near $Z$. Note that $Z$ can be on either side of $Y_n$}
    \label{fig:nintervals}
\end{figure}
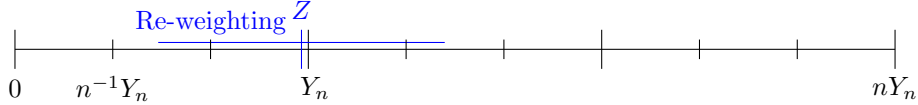

\textbf{Case 2:} $p = p_n - 1$. This case is only non-trivial when $y = 0$, that is, $J = [0,nY_n)$. 
See Figure \ref{fig:nintervals}.
Only two children of $J$ ``see'' the re-weighting: \[ J_1 =  [0,Y_n) \text{ and } J_2 = [Y_n,2Y_n). \] 
Note that $\mu(J_i) = Y_n$ for all $i \geq 3$ (where we abuse notation by using the points $Y_n$ as a value). Noting that $\mu([0,Z)) = Z$, the following is true: 
\begin{align*}
    \vert \mu(J_1) - Y_n \vert &\leq \vert  \mu([0,Z)) - Y_n \vert + \mu([Y_n,Z))  \\
    &= \vert Z - Y_n \vert + \mu([Y_n,Z)) \\
    &\leq \frac{\epsilon}{2^x} + \frac{\epsilon}{2^x}\cdot b^\alpha. 
\end{align*}

The same can be said about $J_2$, where we use $\mu([Z,2Z)) = Z$,
\begin{align*}
    \vert \mu(J_2) - Y_n \vert &\leq \vert \mu([Z,2Z)) - Y_n \vert + \mu([Y_n,Z)] +  \mu([2Y_n,2Z)) \\
    &\leq \frac{\epsilon}{2^x} + \frac{\epsilon}{2^x}\cdot b^\alpha + \frac{2\epsilon}{2^x}.
\end{align*}

Recall that $\epsilon$ is very small (specifically, we defined $\epsilon$ to be $2^{-100\alpha}>0$). 
So, $\epsilon\cdot b^{\alpha}=(2^{-100\alpha})b^{\alpha}\leq 2^{-99}$ must be true for $\alpha\geq 1$. 

For convenience, we bound $\epsilon \cdot b^{\alpha} \leq 0.001$. Then, we have the following:
\[ \frac{\epsilon\cdot b^\alpha}{2^x} \leq 0.001 \cdot Y_n.\]
The same argument holds for both $J_1$ and $J_2$, for which we obtain
\[ \vert \mu(J_1) - Y_n \vert, \vert \mu(J_2) - Y_n \vert \leq 0.001 Y_n \]
which, by another application of the triangle inequality, means that all children of $J$ have measures within a ratio of $1.001^2$. 

\textbf{Case 3:} $p > p_n-1$, or $\vert J \vert \leq Y_n$. 
Now, $J$ cannot cross $Y_n$, and must be on either side of $Y_n$, which is the main reason why this construction works. 
This case will also have two subcases (3.1 and 3.2) to consider. 

\emph{Case 3.1:} $\vert J \vert \leq 2^{-2\alpha}Z$. 
In this case, $\vert J \vert$ is shorter than any nice interval, 
so $J$ intersects at most two nice intervals.
Thus its children's ratios is at most $\kappa^2$. 

\emph{Case 3.2:} $\vert J \vert > 2^{-2\alpha}Z$. This is the only non-trivial case. We further case on whether $J$ has $Y_n$ as an endpoint.

[\emph{3.2.1}] Imagine $J$ does not have $Y_n$ as one of its endpoints. Then $J$ must have at least one sibling between itself and $Y_n$, so the following must hold:
\[ d(J, Z) \geq d(J, Y_n) - d(Y_n, Z) \geq \vert J \vert - 0.1 \vert J \vert = 0.9 \vert J \vert, \]
where $d$ is the Euclidean distance.  Thus, $J$ intersects at most three nice intervals $f$, which results in a factor of still at most $\kappa^3$. 

[\emph{3.2.2}] Now imagine that $J$ has $Y_n$ as one of its endpoints. Without loss of generality, let us suppose that $Y_n$ is its left endpoint. Let the right endpoint of the left-most child of $J$ be $X$.
Let $y$ be the largest such $y$ that satisfies $Z + 2^{-y} < X$ inequality. 

Let the density $f$ on the nice interval $(Z + 2^{-y-1}, Z + 2^{-y})$ be denoted as $f_0$.  Note that this interval is entirely contained in the leftmost child of $J$.  Excluding the leftmost child of $J$, due to the maximality of $y$, this can intersect at most $\log_2(2n) + 2$ nice intervals, which results in a factor of at most $\kappa^{\log_2(2n) + 2}$ (one can refer to the construction in \cite{ABMPZ}, based on the \emph{exhaustion procedure} of \cite{AH} for more details). So on $J$ excluding its leftmost child, we have \[f\leq f_0 \cdot \kappa^{\log_2(2n) + 2}  \] 
Henceforth, for any non-left-most child $J_i$, \[ \mu(J_i) \leq \vert J_i \vert \cdot f_0 \cdot \kappa^2 (2n)^{ \log_2 \kappa} \]

Finally we must address $J$'s left-most child, call it $J_1 = [Y_n, X)$.
We first treat its approximation $[Z,X)$.
Due to our weighting scheme, we know that $\mu$ on $[Z,Z+2^{-y-1})$ has an average value within a factor of $\max(1/a,b) < \kappa$ of the value of $\mu$ on $(Z + 2^{-y-1}, Z + 2^{-y})$, which is $f_0$. Indeed, only on the $y$-th step of the re-weighting do the average values between the two intervals change, 
and all re-weighting after the $y$-th step doesn't change the average value. 
On the other hand, on $[Z + 2^{-y},X)$ the value of $f$ is at most $f_0 \cdot \kappa^{\log_2 (2n) + 2}$, and at least $f_0 / \kappa$.
Therefore, on $[Z,X)$, the average value of $f_0$ is between $f_0 / \kappa$ and $f_0 \cdot \kappa^{\log_2 (2n) + 2}$.  For $[Y_n,X)$, 
by an argument like that in case 2, we use the fact that the $\epsilon$ we chose is very small to argue that 
\begin{equation}
\label{mu measure compare}
     \vert \mu(Y_n, X) - \mu(Z, X) \vert < 0.001 \mu(Y_n, X).
\end{equation} 
Hence, all children of $J$ have measures within a ratio of $1.01 \kappa^3  (2n)^{\log_2 \kappa} $ of each other.

We have assumed that $Y_n$ is the left-endpoint of $J$, but the case with the right endpoint works similarly.
The only difference is looking at $(Z - 2^{-y}, Z - 2^{-y-1})$.
  
To summarize, to prove that $\mu$ is $n$-adic doubling,
we considered a $n$-adic interval $J = (\frac{y}{n^p}, \frac{y+1}{n^p})$ for some $p$, considering possible $p$ to be smaller, equal, and larger than the earlier defined $p_n$.  In all cases, we have shown that the requisite ratio of any two sibling's measure is bounded by  $1.01 \kappa^3  (2n)^{\log_2 \kappa} $.
\end{proof}

\begin{theorem}\label{thm2}
    There exists a measure $\mu$ on $(0,\infty)$ that is $n$-adic doubling for each $n\geq 2$ but not doubling overall.
\end{theorem}

\begin{proof}
Consider some $\eta \in \mathbb{N}$. Theorem \ref{thm1} implies that there exists $\mu$ on $[0,1)$ such that doubling constant $C$ is $\kappa^{\alpha}$, and $n$-adic doubling constant $C(n)$ is bounded by $\kappa^{\alpha}$ for $n>\eta$ and $1.01 \kappa^3  (2n)^{\log_2 \kappa} $ for $n\leq \eta$. 

Now choose $0<a<1$ and $b=2-a$ (so that $1<b<2$); and $\kappa = \frac{b}{a}$
such that $3 \kappa^2 \leq 100$, and finally denote $\gamma = \log_2{\kappa}$.  Choose $\kappa = \kappa_\eta$ and $\alpha = \alpha_\eta$ such that $\kappa^\alpha = \eta^\gamma$ on each $[\eta-1, \eta)$ interval, and define  $\mu = \mu_{(\eta)}$ there.  We know that the doubling constant $C$ is $\kappa^{\alpha}$, and while Theorem \ref{thm1} was for a fixed $\eta$, now we are taking a new $\eta$ each time. Since $\alpha$ depends on $\eta$ (specifically, $\alpha_\eta = \log_{\kappa}{\eta^\gamma}$ is increasing), now the doubling constant $C=\kappa^{\alpha_\eta} = \kappa^{\log_{\kappa}{\eta^\gamma}}=\eta^\gamma$ will become unbounded. Concretely, each iteration has more and more iterative steps, so $C=\kappa^{\alpha}$ is not bounded.  This makes $\mu$ not doubling overall. 

Now we will show that $\mu$ is $n$-adic doubling. Fix some $n$, and consider some $n$-adic interval $J$. We will then have a few cases depending on the size of $J$.

\textbf{Case 1:} $\vert J\vert > 1$.

In this case, all of its children have length at least one, which mean all of its children inherit Lebesgue measure.

We only need to consider small parent intervals. When $J$ has length at most $1$, $J$ cannot contain $n$. So either $J \in [0,n)$ or $J \in [n,\infty)$, introducing two cases.\\

\textbf{Case 2:} $J \in [0,n)$ and $\vert J\vert \leq 1$.

In the case when $J \in [0,n)$, we have $\eta <n$. Which means that for $\mu = \mu_{(\eta)}$ on $[\eta-1, \eta)$, so we can use the result from Theorem \ref{thm1} such that \[ C(n) \leq \kappa^{\alpha_\eta},\text{ which also means that } C(n) \leq \kappa^{\alpha_{n}} \leq n^\gamma \] because $\eta$ is bounded by $n$ and $\alpha_\eta$ is an increasing sequence. Hence $\mu$ is $n$-adic doubling with a constant $n^\gamma$. \\

\textbf{Case 3:} $J \in [n,\infty)$ and $\vert J\vert \leq 1$.

In this case we have $\eta \geq n$, and we again can use the result from Theorem \ref{thm1} for $\mu = \mu_{(\eta)}$ on $[\eta-1, \eta)$ such that \[ C(n) \leq 1.01 \kappa^3 (2n)^{ \log_2 \kappa} = 1.01\kappa^3(2n)^\gamma. \]  Hence $\mu$ is guaranteed to be $n$-adic doubling with a constant $100(2n)^\gamma$ due to the assumption on the size of $\kappa$.\\

Combining everything, we can conclude that $\mu$ is $n$-adic doubling but not doubling overall on $(0,\infty)$.\\
\end{proof}

Note that we can actually strengthen the result so that the $n$-adic doubling constant
is sub-linear in $n$. The $1.01\kappa^3(2n)^\gamma$ part already satisfies this. On $[\alpha - 1, \alpha)$, we need not redistribute $\alpha$ steps, but we can instead redistribute $\gamma \log_\kappa \alpha$ steps so that the $f$ ratio is at most $\alpha^\gamma$. 

Finally, let's prove the final step of our main result.

\begin{theorem}\label{thm3}
    Let $f(n) > 1$ be an increasing and unbounded function.
    There exists an infinite family of measures $\mu$ on $[0,1)$ that 
    are $n$-adic doubling with $n$-adic constant at most $f(n)$ for each $n \geq 2$, but not doubling.
\end{theorem}

\begin{proof}
    We do the redistribution entirely within $[0,1)$. 
    For $\eta \in \mathbb{N}$, 
    on step $\eta$, we will 
    redistribute the uniform Lebesgue measure on $[0,1)$
    to that of $\mu^{(\eta)}$, which is the measure 
    obtained from \ref{thm1} with parameters $\kappa_\eta$, $\alpha_\eta$, and $M_\eta$. We will choose the exact values later,
    but for now we assert that $M_\eta$ is increasing in $\eta$. 

    The redistribution occurs on $[Z^{(\eta)}/2, 3Z^{(\eta)}/2)$
    for some $Z^{(\eta)}$. There are infinitely many, arbitrarily
    small possible values of $Z^{(\eta)}$. For all
    $n \leq M_\eta$, there exists an $n$-adic point
    $Y_n^{(\eta)}$ very close to $Z^{(\eta)}$. We may define
    \[ L_n^{(\eta)} = [0, n^{-1} Y_n^{(\eta)}) \text{ and } R_n^{(\eta)} = [0, n Y_n^{(\eta)}). \]
    Note that the redistribution is contained on $R_n^{(\eta)} \setminus L_n^{(\eta)}$. Given $Z^{(\eta)}$, we choose $Z^{(\eta+1)} \leq (M^{(\eta)})^3$.
    This guarantees that the redistributions are sufficiently separated
    in the sense that $R_n^{(\eta+1)}  \subset L_n^{(\eta)}$,
    where the containment is strict. 

    After the measure has been redistributed for all $\eta \in \mathbb{N}$,
    we know that the doubling constant of $\mu$ satisfies 
    \[ C \geq \sup_{\eta \geq 1} (\kappa_\eta)^{\alpha_\eta} \]

    Easily $\mu$ is dyadic doubling with constant $\kappa$. 
    Fix $n \geq 3$. Let $\eta_n$ be the minimum $\eta$
    such that $n \leq M_\eta$, and for all $\eta \geq \eta_n$, define $Y_n^{\eta}$, $L_n^{\eta}$ and $R_n^{\eta}$.
    Note that 
\[ \ldots \subset R_n^{(\eta_n+2)} \subset L_n^{(\eta_n+1)} \subset R_n^{(\eta_n+1)}  \subset L_n^{(\eta_n)} \subset R_n^{(\eta_n)} \subset [0,1) \]
is an infinite strictly descending chain of $n$-adic intervals all with endpoint zero.  

Let $J$ be an arbitrary $n$-adic interval and $J_1, J_2$ arbitrary children of $J$.
Let the minimum element of the chain above that contains $J$ be $I$. 
Note that neither of $J_1, J_2$ can ``see'' any redistribution
due to a smaller element in the chain, since all containments are strict. 

If $I$ is some $L_n^{(\eta)}$
then $J$ sees no redistribution, so the ratio of $\mu(J_1),\mu(J_2)$ is $1$.
If $I$ is $[0,1)$, then $J$ can see $\mu^{(\eta)}$ for all
$\eta < \eta_n$, so
the doubling constant is at most
\[ \sup_{\eta < \eta_n} (\kappa_\eta)^{\alpha_\eta}. \]
Finally, if $I$ is some $R_n^{(\eta)}$, then 
$\mu = \mu^{(\eta)}$ on $J$ for some $\eta \geq \eta_n$, and the doubling
constant is at most 
\[ \sup_{\eta \geq \eta_n} 1.01 (\kappa_\eta)^{3 + \log_2 2n}\]
Hence, we have that the $n$-adic doubling constant satifsies
\[ C(n) \leq \max \left( \sup_{\eta \geq \eta_n} 1.01 (\kappa_\eta)^{3 + \log_2 2n},
\sup_{\eta < \eta_n} (\kappa_\eta)^{\alpha_\eta} \right) := \max(A_n, B_n) \]

Finally we pick $\alpha_n, \kappa_n$, and $\eta_n$. Let $\kappa_n > 1$ be such that 
\[ 1.01 (\kappa_n)^{3 + \log_2 2 \eta_n} = f(n). \]
Note that we have used $f(\eta) > 1.01$ (but we can improve the constant $1.01$ to any constant $> 1)$. 
Without loss of generality, we can assume that $f$ is sufficiently slow growing so that $\kappa_\eta$ is decreasing.
So 
\[ A_n = \sup_{\eta \geq \eta_n} 1.01 (\kappa_\eta)^{3 + \log_2 2n} \leq 1.01 (\kappa_n)^{3 + \log_2 2n} = f(n). \]

Let us next take $\alpha_\eta = 3 + \log_2 (2\eta)$. Therefore, 
$(\kappa_\eta)^{\alpha_\eta} = f(\eta) / 1.01$, and $C = \sup_{\eta \geq 1}
f(\eta) = \infty$. 

Finally, define $\eta_n$ such that 
\[ B_n = \sup_{\eta < \eta_n} (\kappa_\eta)^{\alpha_\eta}
\leq \sup_{\eta < \eta_n}  (\kappa_1)^{3 + \log_2(2\eta)}
\leq (\kappa_1)^{3 + \log_2(2\eta_n)} \leq f(n). \]
Since $f(n)$ is increasing, $\eta_n$ is unbounded and $M_\eta$ is thus well-defined,
we have that $B_n \leq f(n)$. 
We have therefore shown that $\mu$ is $n$-adic doubling with constant $f(n)$, yet not doubling overall.    
\end{proof}

    \section{Applications to weight and function classes}
We begin with a few definitions and conclude with the proofs of all the claims in Theorem \ref{app}.
\subsection{Definitions}

The measures we have defined automatically also define an associated weight.  We begin with defining some distinguished weight classes: reverse H\"older weights $RH_r$ and Muckenhoupt weights $A_r$.
\begin{definition}
     Let $1 < r < \infty$. Given a \emph{weight} (a non-negative, locally integrable) $w$, we say $w \in RH_r$, the $r$-reverse Hölder class, if
     \[ (\fint_I w^r)^\frac{1}{r}\leq C\fint_I w. \] 
     for all intervals $I$, where $C$ is an absolute constant that may
     depend on $w$ but not on the interval $I$. Moreover, we say $w \in RH_1$ if $w\in RH_r$ for some $r>1$; that is 
\[ RH_1=\bigcup_{r>1}RH_r. \]
    \end{definition}

    \begin{definition}
    Let $1 < r < \infty$.  Given a weight $w$, we say a weight $w\in A_r$, the Muckenhoupt $A_r$ class if 
    \[ \sup_I (\fint_I\;w(x)dx)(\fint_I w(x)^\frac{-1}{r-1}dx)^{r-1}<\infty. \]
    where the supremum is taken over all intervals $I$. Moreover we say $w \in A_{\infty}$ if $w \in A_r$ for some $r>1$; that is
\[ A_1=\bigcup_{r>1}A_r. \]
    \end{definition}

Recall the definitions of the space of functions of bounded mean oscillation $BMO$:

    \begin{definition}
            We say a function is in $BMO$ if and only if \begin{align*}
        ||f||_{BMO} \coloneqq \sup_I \fint_I |f - \fint_I f|< \infty. 
    \end{align*}
    \end{definition}

Next we discuss the space of \emph{vanishing mean oscillation}, or $VMO$.  Essentially these functions are ones who approach their averages over small and large intervals.  As such, there are slightly different definitions for the space over the real line and the torus.  We recall these below, but will actually only use the important duality relationships to pass from $BMO$ to $VMO$ via the Hardy space $H^1$.

\begin{definition}
    A function $f$ in $BMO$ is said to be in $VMO$ if  it satisfies the following conditions:
    \begin{enumerate}
        \item $\lim_{\delta \rightarrow 0} \sup_{I: |I|< \delta}\fint_I |f-\fint_If| = 0$;
        \item $\lim_{N \rightarrow \infty} \sup_{I: |I|> N}\fint_I |f-\fint_If| = 0$; and 
        \item $\lim_{R \rightarrow \infty} \sup_{I: I\cap B(0,R) = \emptyset}\fint_I |f-\fint_I f| = 0$.
    \end{enumerate}
    where $B(0,R)$ is the interval centered at zero of radius $R$.
\end{definition}
The definition on the torus is the same, but without the second and third conditions.

\begin{definition}
    The Hardy space $H^1$ as the dual of $VMO$.  Additionally, the predual of $H^1$ is $BMO$.  That is:
    \begin{enumerate}
        \item $H^{1*} = BMO$
        \item $VMO^* = H^1$.
    \end{enumerate}
\end{definition}
Originally $H^1$ was defined by other means and the duality relationships were shown by Fefferman-Stein \cite{FS} and Coifman-Weiss \cite{CW}.  But for our purposes, this definition will be useful.

    We extend all definitions above in the obvious way to 
    an $n$-adic system, by only considering $n$-adic intervals,
    to obtain spaces $RH_n^r$, $A^n_r$, $BMO_n$, $H^1_n$ and $VMO_n$, with one minor caveat.  Since extending the definition of $RH^r$ to $RH_n^r$ in the natural way can create a weight that may not be $n$-adic doubling, we additionally require $RH_n^r$ weights to be $n$-adic doubling.  Also note that the dual of $VMO_n$ is $H^1_n$ and the dual of $H^1_n$ is $BMO_n$. 

\subsection{Proof of Theorem \ref{app}}
    Now we will utilize the measure we constructed in the previous
    section to show the first claim in Theorem \ref{app}, via focusing on the weight $f$ defined by our measure $\mu$, that is,
\[
\mu(I)=\int_I f_\mu dx \quad \textrm{for any interval} \ I. 
\]
  
\begin{lemma}
\label{RH aux lemma}
    Let $f : \Omega \rightarrow [0,\infty)$ be integrable, let $C(I)$ denote
    the reverse Holder constant of $f$ on $I$. Suppose that $I_1,\ldots,I_m$ are disjoint with $I = \bigcup_{n=1}^m I_n$, 
    then 
    \[ C(I) \leq m^{1/r} \max_n C(I_n) \left( \frac{\vert I \vert}{\vert I_n \vert} \right)^{1-1/r}
    \leq \max_n C(I_n) \cdot \max_n \frac{\vert I \vert}{\vert I_n \vert}. \] 
    If however the averages of $f$ on all $I_n$ are equal, then
    \[ C(I) \leq \max_n C(I_n). \]
\end{lemma}

\begin{proof}
    \begin{align*}
        \fint_I f_n^r dx &= \sum_n \frac{\vert I_n \vert}{\vert I \vert} \fint_{I_n} f_n^r dx \\
        &\leq   \sum_n \frac{\vert I_n \vert}{\vert I \vert} C_n^r \left( \fint_{I_n} f dx \right)^r \\
        &= \sum_n \frac{\vert I_n \vert^{1-r}}{\vert I \vert}  C_n^r \left( \int_{I_n} f dx \right)^r \\
        &\leq \sum_n \frac{\vert I_n \vert^{1-r}}{\vert I \vert}  C_n^r \left( \int_{I} f dx \right)^r \\
        &= \sum_n \frac{\vert I_n \vert^{1-r}}{\vert I \vert^{1-r}}  C_n^r \left( \fint_{I} f dx \right)^r \\
        &\leq m \max_n C_n^r \left(  \frac{\vert I \vert}{\vert I_n \vert} \right)^{r-1} \left( \fint_{I} f dx \right)^r 
    \end{align*}
    Using the fact that 
    \[
     m^{1/r} \leq \max_n \bigg(\frac{\vert I \vert}{\vert I_n \vert}\bigg)^{1/r}
    \]
    finishes the claim.
    
    We now show the second claim. If $\fint_{I_n} f = A$ for all $I_n$, then
    \[ \fint_{I_n} f^r \leq C(I_n)^r A^r \text{ if and only if} 
    \int_{I_n} f^r \leq C(I_n) A^r \vert I_n \vert. \]
    Thus, 
    \[ \sum_{i=1}^n \int_{I_n} f^r \leq C(I_n) A^r \vert I \vert \text{ if and only if} \fint_{I} f^r \leq \max C(I_n)^r A^r, \]
    as desired. 
\end{proof}

\begin{theorem}
\label{RH thm}
     Let $1 \leq r < \infty$. There exists a weight $f : [0,1) \rightarrow (0,\infty)$ such that $f \in RH_n^r$
     for all $n \geq 2$ but $f \notin RH^r$. 
\end{theorem}

\begin{proof}
    
The basic idea of the proof has similarities to applications proved in \cite{AH}, but the intricacies involved in constructing our measure necessitate some care.  Fix $r > 1$. We use the density $f$ of the measure from Theorem \ref{main} with the property that $C(n) \leq 1.01 \kappa^3 (2n)^{ \log_2 \kappa}$. 
Additionally, we pick $\kappa_1$ such that $(b_1)^r < 2$.
Since $f$ is not doubling, necessarily $f \notin RH^r$ for any $r \geq 1$. We then prove $f \in RH_n^r$
for all $n \geq 2$.  Note that we have already proved the $n$-adic doubling part of this assertion.

Recall that
\[ \ldots \subset R_n^{(\eta_n+2)} \subset L_n^{(\eta_n+1)} \subset R_n^{(\eta_n+1)}  \subset L_n^{(\eta_n)} \subset R_n^{(\eta_n)} \subset [0,1) \]
is an infinite strictly descending chain of $n$-adic intervals all with endpoint zero.  

Define
\[ S_0 = [0,1) \setminus R_n^{\eta_n}, 
    S_{2k+1} =  R_n^{(\eta_n + k)} \setminus L_n^{(\eta_n + k)}, 
     S_{2k+2} =  L_n^{(\eta_n + k)} \setminus R_n^{(\eta_n + k + 1)}. \]
Thus, $S_j$ is a the sequence of successive differences of the infinite descending chain,
and $S_j$ partitions $[0,1)$. 
More importantly, the average of $f$ is $1$ on all of these intervals, and $f$ is identically $1$
on $S_{2k}$ for $k \geq 1$. 

Consider an arbitrary $n$-adic interval $J$. Define $C(J)$ to be the $RH_n^r$ constant of $J$. 
We wish to show that $C(J)$ is bounded. There are two cases. Either $J$ is contained in some $S_k$ or it
has left endpoint $0$ and is sandwiched between two intervals of the infinite descending chain.
We thus have six total cases.  We list the first three.

\begin{enumerate}
    \item $J \subseteq S_{2k+1}$ for $k \geq 0$. We claim that $C(J) \leq C$ for some constant $C$
    that only depends on $b$ and $r$. We will temporarily assume this and prove this later.
    \item $J \subseteq S_{2k+2}$ for $k \geq 0$. Here $C(J) = 1$.
    \item $J \subseteq S_0$. Since the ratio of possible values of $f$ on $S_0$ is bounded 
    by $1.01 \kappa^3 (2n)^{ \log_2 \kappa}$, we can say $C(J) \leq 1.01 \kappa^3 (2n)^{ \log_2 \kappa}$.
\end{enumerate}

Note that the reverse H\"older constant can also be defined for a union of $n$-adic intervals,
in particular, $S_{j}$. Note that $C(S_{0}) \leq 100 n$, $C(S_{2k+2}) = 1$,
and since $S_{2k+1}$ can be written as a union of $2n-2$ many $n$-adic intervals,
each of constant $C$, and length at least $\vert S_{2k+1} \vert / n^2$, by the lemma we 
can conclude that $C(S_{2k+1}) \leq C n^2$. 

We now list the last three cases, where $J$ is sandwiched between two intervals of the chain.

\begin{enumerate}
\setcounter{enumi}{3}
    \item $R_n^{(\eta_n + k + 1)} \subseteq J \subseteq L_n^{(\eta_n + k)}$. Partition $J$ as 
    \[ J = (J \setminus R_n^{(\eta_n + k + 1)} ) \cup \bigcup_{j=2k+3}^\infty S_{j}. \]
    Since $f$ averages to $1$ on each of these intervals, and the reverse H\"older constant is at most $C n^2$
    for these intervals, we can conclude by the lemma that $C(J) \leq C n^2$. 
    \item $L_n^{(\eta_n + k)} \subset J \subset R_n^{(\eta_n + k)}$ for some $k \geq 0$. Note that 
    necessarily $J = [0,Y_n^{(\eta_n+k)})$. Write $J$ as a union of $n$ many $n$-adic intervals (its children).
    Since each of these intervals have constant at most $Cn^2$, and each interval has length at least $\vert J \vert / n$,
    by the lemma we can conclude that $C(J) \leq C n^3$. 
    \item $R_n^{(\eta_n)} \subseteq J \subseteq [0,1)$. Write $J = (J \setminus R_n^{(\eta_n)}) \cup R_n^{(\eta_n)}$. 
    We have
    \[ C(J \setminus R_n^{(\eta_n)}) \leq 100 n \text{ and } C(R_n^{(\eta_n)}) \leq C n^2. \]
    Additionally, $\vert R_n^{(\eta_n)} \vert \leq \vert J \setminus R_n^{(\eta_n)} \vert$, so 
    \[ C(J) \leq \frac{\max (1.01 \kappa^3 (2n)^{ \log_2 \kappa}, Cn^2)}{\vert R_n^{(\eta_n)} \vert}.  \]
    Note that $\vert R_n^{(\eta_n)} \vert$ is a fixed constant. 
\end{enumerate}
 
Thus, assuming that $C(J) \leq C$ for all intervals $J$ in the first case, 
we concluded that the reverse H\"older constant of intervals of all cases is bounded.
We now treat the first case and we will drop $\eta$ since $J \subseteq R_n^{(\eta)} \setminus L_n^{(\eta)}$ for some $\eta$. Recall Figure \ref{fig:nintervals2}. 

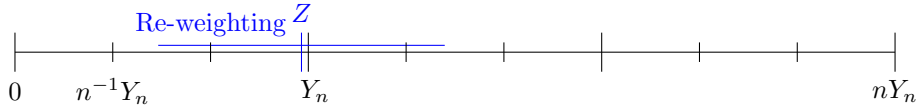
\begin{figure}[!ht]
    \centering
        \begin{tikzpicture}[scale=1.3]
\def\e{0.1} 
\def\eps{0.07} 
\def\f{0.4} 
\draw (0,0) -- (9,0);
\draw (0,2*\e) -- (0,-2*\e);
\draw (9,2*\e) -- (9,-2*\e);
\draw (3,2*\e) -- (3,-2*\e);
\draw (6,2*\e) -- (6,-2*\e);
\draw[blue] (3-\eps,2*\e) -- (3-\eps,-2*\e);
\draw[blue] (1.5-0.5*\eps,\eps) -- (4.5-1.5*\eps,\eps);
\draw (1,\e) -- (1,-\e);
\draw (2,\e) -- (2,-\e);
\draw (4,\e) -- (4,-\e);
\draw (5,\e) -- (5,-\e);
\draw (7,\e) -- (7,-\e);
\draw (8,\e) -- (8,-\e);
\draw (0,-\f) node {$0$};
\draw (9,-\f) node {$nY_n$};
\draw (3+\eps,-\f) node {$Y_n$};
\draw (1,-\f) node {$n^{-1}Y_n$};
\draw[blue] (2,0.7*\f) node {Re-weighting};
\draw[blue] (3-\eps,\f) node {$Z$};
\end{tikzpicture}
    \caption{$n$-adic intervals near $Z$ where $n=3$. Note that $Z$ can be on either side of $Y_n$}
    \label{fig:nintervals2}
\end{figure}

If $\vert J \vert \leq 2^{-2\alpha} Z$, then $J$ crosses at most one value change of $f$,
so $C(J) \leq \kappa$.  Thus we assume $\vert J \vert > 2^{-2\alpha} Z$. 
If $J$ does not have $Y_n$ as an endpoint, then $J$ crosses at most two changes of values of $f$, so $C(J) \leq \kappa^2$.

Now let us assume finally that both $\vert J \vert > 2^{-2\alpha} Z$ and $J$ has $Y_n$ as an endpoint. 
Once again, without loss of generality, $J$ lies to the right of $Y_n$. It will be clear our proof also works when $J$ lies on the left.  We are now in the setup as in Case 3.2.2 in Theorem \ref{thm1}.  Let $X$ be the right endpoint of the leftmost child of $J$.

Let $y$ be the smallest step such that $Z + 2^{-y} \geq X$.
Suppose that $f = f_0$ on $[Z +  2^{-y-1}, Z +  2^{-y})$. On this interval, $f_0$
can be anything between $a^\alpha$ and $b^\alpha$. 
By construction, we have 
\[ a f_0 \leq \fint_Z^X f dx \leq b f_0. \]

By \eqref{mu measure compare}, we have that $\mu(Y_n,Z) \leq 0.001\mu(Y_n,X)\leq 0.01\mu(X,Z)$.  Thus also $|Y_n-Z| \leq 0.01|X-Z|$ and $0.99|J| = 0.99(|Y_n-Z|+|X-Z|) \leq |X-Z|$, and we can extend the average integral over $(X,Z)$ to $J$ within a factor of 0.01.  That is,
\[ 0.99 a f_0 \leq \fint_J f dx \leq 1.01 b f_0. \]
We will use this lower bound shortly.  Now let us focus our attention on $[Z,Z + 2^{-y-1})$. 
By definition of $f_0$, regardless of where we are
in our weighting scheme we know that 
$f$ is at most $\kappa b f_0$ on $[Z + 2^{-y-2}, Z +  2^{-y-1})$,
at most $\kappa b^2 f_0$ on $[Z + 2^{-y-3},Z + 2^{-y-2})$,
and so on. Thus,
\[ \fint_Z^{Z +  1/2^{y+1}} f^r dx 
   \leq \frac{b^r \kappa^r}{2} f_0^r +  \frac{b^{2r} \kappa^r}{4} f_0^r  + \frac{b^{3r} \kappa^r}{8} f_0^r  + \ldots
   \leq \frac{b^r \kappa^r}{2-b^r} f_0^r 
   \]
Here it is important that $b^r < 2$. Thus, by the triangle inequality
\begin{align*}
    \fint_J f^r dx &\leq \frac{\vert Y_n - Z \vert}{\vert J \vert} \vert \fint_{Y_n}^Z f^r dx \vert \\
&+  \frac{\vert 1/2^{y+1} \vert}{\vert J \vert}
\fint_Z^{Z +  1/2^{y+1}} f^r dx \\
&+\frac{\vert X - Z - 1/2^{y+1} \vert }{\vert J \vert} \fint_{Z +  1/2^{y+1}}^{X} f^r dx
\end{align*}

We have previously upper bounded the second average integral.
The third average integral has value exactly $f_0^r$. 
Moreover, the coefficients of the second and third integral are no greater than $1.001$ and $0.501$ respectively.
To see this, note that 
\[ \frac{1}{2^{y+1}} \leq \vert X - Z \vert \leq 1.001 \vert J \vert
\text{ and } \vert X - Z - \frac{1}{2^{y+1}} \vert \leq 0.5 \vert X - Z \vert \leq 0.501 \vert J \vert \]
where we have used that $y$ is maximal. 

The first average integral 
has value $(a^\alpha b^\alpha)^r \leq a^{\alpha r} (b^r)^\alpha \leq f_0^r 2^\alpha$. But the coefficient in front of it is small as can be seen by
\[ \frac{\vert Y_n - Z \vert}{\vert J \vert} \vert \fint_{Y_n}^Z f^r dx \vert 
\leq \epsilon \frac{Z}{\vert J \vert} \cdot f_0^r 2^\alpha \leq \frac{2^{-100\alpha}}{2^{-2\alpha}}  f_0^r 2^\alpha \leq 0.001 f_0^r, \]

where we have used the fact that $\vert J \vert > 2^{-2\alpha} Z$ in this case.
Putting the three terms together, we obtain
\[ \fint_J f^r dx \leq 0.001 f_0^r + 1.001 \cdot \frac{b^r \kappa^r}{2-b^r} f_0^r + 0.501 f_0^r. \]

Therefore, we can conclude that
\[ C(J) \leq C := \max ( \left( \frac{\frac{1.001b^r \kappa^r}{(2-b^r)} + 1}{0.99a} \right)^{1/r} , \kappa^2) \]
which is independent of both $J$ and $n$, and proves our claim.
\end{proof}

The proof of the rest of the claims in Theorem \ref{app} mostly follow from a similar argument as the corresponding theorem in \cite{ABMPZ}.  This argument starts with a key example and weaves it through the weight and function classes $A_r$, $BMO$, $H^1$ and $VMO$ using functional analysis and deep theorems such as duality.  We include the short, self-contained argument of the other assertions of Theorem \ref{app}.
\begin{proof}
    
    Let $f$ be the density of $\mu$ constructed
    for \ref{main}. Depending on the $r$ of our interest,
    we pick $\kappa_\eta$ not too large 
    so that $b^r < 2$ for all $\eta$ steps. Therefore, $f \in RH_n^r$ for all $n$.
    However, $f$ is not doubling, so it cannot 
    be in $RH^r$ for any $r > 1$. 
    The case for $A_n^r$ and $A^r$ is the same,
    but one needs to interchange the roles of $a$ and $b$, and replace $r$ with $\frac{-1}{r-1}$.  See \cite{AH} for details.  If we want the weaker conclusion about $A_\infty$, this is implied automatically by the reverse H\"older results, and can be used to imply the next three results on $BMO$, $H^1$, and $VMO$, as we detail below.

    It is a known fact that if $f \in A_\infty^n$, then $\log \vert f \vert \in BMO_n$.  Hence, $\log \vert f \vert \in \bigcap_{n \geq 2} BMO_n$ However, again since $f$ is not doubling, we can easily
    check that $\log \vert f \vert \notin BMO$. 

    The last two results follow from duality. Given
    Banach spaces $X, X_1, X_2, \ldots$, if 
    $\bigcap_n X_n^* \neq X^*$, then $\sum_n X_n \neq X$. 
    Likewise if $\sum_n X_n^* \neq X^*$, then $\bigcap_n X_n \neq X$.  These two facts can be used exactly as in \cite{ABMPZ} to give us the last two parts of Theorem \ref{app}.
\end{proof}



\begin{thebibliography}{ABCD99}

\bibitem{ABMPZ}  T.C. Anderson, E. Bellah, Z. Markman, T. Pollard and J. Zeitlin.  Arbitrary finite intersections of doubling measures and applications.  To appear in \emph{Journal of Functional Analysis}.

\bibitem{AH} T. C Anderson and B. Hu, A Structure Theorem on Intersections of General Doubling Measures and Its Applications, International Mathematics Research Notices, Volume 2023, Issue 9, May 2023, Pages 7423–7485.

\bibitem{ATV}  T.C. Anderson, C. Travesset, J. Veltri. The Structure of Weight and Function Classes With Coprime Bases. The Quarterly Journal of Mathematics (2021).

\bibitem{Sch}
James Ax. On Schanuel's Conjectures. \emph{Annals of Mathematics}, Second Series, Vol. 93, No. 2 (Mar., 1971), pp. 252-268 (17 pages). 

\bibitem{BMW} D.M Boylan, S.J. Mills, and L.A. Ward. Construction of an exotic measure: dyadic doubling and triadic doubling does not imply doubling. \emph{J. Math. Anal. Appl.} 476 (2019), no. 2, 241--277.

\bibitem{Cruz} D. Cruz-Uribe.  Function spaces, embeddings and extrapolation X, Paseky.  Matfyzpress, Charles University, 2017.

\bibitem{CN} D. Cruz-Uribe; C.J. Neugebauer. The structure of the reverse Hölder classes. Trans. Amer. Math. Soc. 347 (1995), no. 8, 2941--2960. 

\bibitem{CW} R. Coifman and G. Weiss.  Extensions of Hardy spaces and their use in analysis  Bull. Amer. Math. Soc. 83 (1977), no. 4, 569–645.

\bibitem{FS} C. Fefferman and E.M. Stein.  $H_p$ spaces of several variables.  Acta Math. 129(1972), no.3-4, 137–193.

\bibitem{GJ} J. B. Garnett and P. W. Jones. BMO from dyadic BMO. Pacific J. Math., 99(2):351–371, 1982.

\bibitem{G} F. W. Gehring,  The $L^p$ integrability of partial derivatives of a quasiconformal mapping, Acta. Math.130 (1973), 265–277

\bibitem{Grafakos} L. Grafakos.  Modern Fourier analysis, Third edition. Grad. Texts in Math., 250. Springer, New York, 2014.

\bibitem{Krantz} S.G. Krantz. On functions in $p$-adic BMO and the distribution of prime integers. \emph{J. Math. Anal. Appl.} 326 (2007), no. 2, 1437--1444.

\bibitem{LPW} J. Li, J. Pipher, L.A. Ward, Dyadic structure theorems for multiparameter function spaces, \emph{Rev. Mat. Iberoam}. {\bf 31} (2015), no. 3, 767--797.

\bibitem{P} C. Pereyra. Dyadic harmonic analysis and weighted inequalities: the sparse revolution.
In: Aldroubi A., Cabrelli C., Jaffard S., Molter U. (eds) New Trends in Applied Harmonic Analysis, Volume 2. Applied and Numerical Harmonic Analysis. Birkhauser, Cham (2019) 259--239. Available at arXiv:1812.00850v1.

\bibitem{PWX} J. Pipher, L.A. Ward, and X. Xiao. Geometric-arithmetic averaging of dyadic weights. Rev. Mat. Iberoam. 27 (2011), no. 3, 953--976.

\bibitem{S}
D. Sarason.  Functions of vanishing mean oscillation.  \emph{Trans. Amer. Math. Soc.} 207(1975), 391–405.

\bibitem{TM} T. Mei, $BMO$ is the intersection of two translates of dyadic $BMO$. \emph{C.R. Acad. Sci. Paris, Ser}. I {\bf 336} (2003), 1003--1006.

\bibitem{Ward} L.A. Ward.  Translation averages of dyadic weights are not always good weights.
Rev. Mat. Iberoamericana 18(2): 379-407 (June, 2002). 
\end{thebibliography}
\end{document}